\documentclass[a4paper]{article}

\usepackage{graphicx}
\usepackage[utf8]{inputenc}
\usepackage{amsmath, amsthm, amssymb, amsfonts,mathtools}

\usepackage{fourier}
\usepackage[french,english]{babel}
\usepackage{microtype}
\usepackage[backend=bibtex,style=alphabetic]{biblatex}

\newtheorem{theorem}{Theorem}[section]
\newtheorem{lemma}[theorem]{Lemma}
\newtheorem{proposition}[theorem]{Proposition}
\newtheorem{corollary}[theorem]{Corollary}
\newtheorem{definition}[theorem]{Definition}

\newtheorem*{proposition*}{Proposition}
\newtheorem*{lemma*}{Lemma}
\newtheorem*{theorem*}{Theorem}
\newtheorem*{corollary*}{Corollary}
\newtheorem*{conjecture*}{Conjecture}

\title{Affine three-manifolds with centralizing holonomy}

\author{Raphaël V.~{\sc Alexandre}
\footnote{
Université de Strasbourg, IRMA, 7 rue René Descartes, 67000 Strasbourg.
Email address: {\tt raphael.alexandre@math.cnrs.fr}.
This work of the Interdisciplinary Thematic Institute IRMIA++, as part of the ITI 2021 - 2028 program of the University of Strasbourg, CNRS and Inserm, was supported by IdEx Unistra (ANR-10-IDEX-0002), and by SFRI - STRAT’US project (ANR-20-SFRI-0012) under the framework of the French Investments for the Future Program. 
}
}

\newcommand\dd{\,{\rm d}}

\DeclareMathOperator*\id{id}
\newcommand\R{\mathbf{R}}

\DeclareMathOperator*\GL{GL}

\DeclareMathOperator*\SL{SL}

\newcommand\cE{\mathcal{E}}

\newcommand{\iI}{\mathopen{[}0\,,1\mathclose{]}}

\bibliography{ref.bib}

\begin{document}
\maketitle

\begin{abstract}
According to the Markus conjecture, closed flat affine manifolds with parallel volume should be complete. We show it is the case for three-manifolds when the holonomy centralizes an affine transformation preserving the volume. It is notably the case when the holonomy group has non trivial center or when the automorphism group is non discrete.
\end{abstract}

\section*{Introduction}

A smooth manifold $M$ is a \emph{flat affine manifold} when there is a pair $(D,\rho)$ of a developing map $D\colon \widetilde M \to \R^n$ which is a local diffeomorphism, with $\widetilde M$ the universal cover of $M$, and $\rho\colon\pi_1(M,p)\to \R^n\rtimes \GL(n,\R)$ a holonomy morphism verifying the equivariance condition: 
\begin{equation}
\forall x\in\widetilde M,\; \forall g\in\pi_1(M,p),\; D(gx)=\rho(g)D(x).
\end{equation}
The image of the holonomy morphism is called the \emph{holonomy group}.

A manifold $M$ with this structure has \emph{parallel volume} when the holonomy morphism takes its values in $\R^n\rtimes \SL(n,\R)$.
A flat affine manifold $M$ is \emph{complete} when $D\colon\widetilde M \to \R^n$ is a diffeomorphism. One of the main conjectures about closed flat affine manifolds is:

\begin{conjecture*}[Markus~\cite{Markus}]
Any connected closed flat affine manifold with parallel volume is complete.
\end{conjecture*}

It is still an open question, even in dimension three.
Most of the cases known about the Markus conjecture consist in adding an hypothesis on the holonomy group. Notably, Fried-Goldman-Hirsch~\cite{FGH} have shown the case with nilpotent  holonomy, generalizing the abelian case shown by Smillie~\cite{Smillie2}.
Also, Carrière~\cite{Carriere} has shown the conjecture when the holonomy has discompacity one. 

This last hypothesis means that for any diverging sequence of holonomy transformations, at most \emph{one} singular value of the linear parts of the transformations can tend to zero. In particular, Carrière deduces the completeness of every flat Lorentzian manifold.

More diverse hypothesis were studied to prove the conjecture. When the holonomy belongs to a small solvable group it holds~\cite{GH2}, when the developing map is sufficently constrained it can also hold, see~\cite{JK}, \cite{Tholozan}.
The author has shown~\cite{these} completeness of manifolds when the holonomy  preserves a class of foliations.

However, Markus conjecture is only known in its general statement in dimension $2$ by Carrière theorem and complex dimension $2$ with tools from complex geometry (see~\cite{Klingler}).

Our theorem is in dimension $3$ and states a hypothesis on the holonomy group. We will comment this hypothesis afterward.

\begin{theorem*}[\ref{thm-3fold}]
Let $M$ be a connected closed $(\R^3,\R^3\rtimes \SL(3,\R))$-manifold. If the holonomy of $M$ centralizes a non trivial transformation $\phi\in \R^3\rtimes \SL(3,\R)$, then $M$ is complete.
\end{theorem*}

The proof of the theorem consists in proving that the holonomy factorizes through the subgroup:
\begin{equation}
\R^3\rtimes(\R^2\rtimes\SL(2,\R)) = \R^3\rtimes
\begin{pmatrix}
1 & \begin{matrix}\R & \R \end{matrix}\\\begin{matrix} 0\\0\end{matrix} &\SL(2,\R)\end{pmatrix}
\end{equation}
and that all closed  $(\R^3\rtimes(\R^2\rtimes\SL(2,\R)),\R^3)$-manifolds are complete. It is a new completeness statement since \emph{a priori} the holonomy group is  not nilpotent (nor solvable) and might have discompacity $2$ in the sense of Carrière.

Applications of this theorem follow from  two main interpretations of the hypothesis.

\begin{corollary*}[\ref{cor-1}] 
Let $M$  be a  connected closed $(\R^3,\R^3\rtimes \SL(3,\R))$-manifold. 
\begin{itemize}
\item If the holonomy of $M$ has non trivial center,   then $M$ is complete.
\item If the automorphism group of $M$ is non discrete, then $M$ is complete.
\end{itemize}
\end{corollary*}

The first case of a non trivial center happens when the manifold is a Seifert bundle over an orientable surface.\footnote{Note that by \cite{GH} and \cite{Smillie}, the holonomy of the fiber never vanishes when the flat affine structure has parallel volume.} We deduce  the following corollary which was known by  different methods following \cite{GH2} and \cite{CDM}.

\begin{corollary*}[\ref{cor-2}]
Let $M$ be a connected closed  Seifert  bundle over an orientable surface. Then $M$ verifies the Markus conjecture.
\end{corollary*}

\paragraph{Note}
Fried and Goldman~\cite{FG} have classified \emph{complete} affine three-manifolds.
So one may wonder if our hypothesis is credible regarding this classification. In fact,
up to finite index any complete  affine three-manifold verifies our hypothesis. 

Indeed, Fried and Goldman describe three families of subgroups of $\R^3\rtimes \SL(3,\R)$ where the holonomy takes its values, up to a finite cover. The three families are:
\begin{equation}
\R^3\rtimes \begin{pmatrix} 1 & \star & \star \\ 0 & 1 & \star \\ 0 & 0 & 1\end{pmatrix}, \; \R^3\rtimes \begin{pmatrix} 1 & \star & \star \\ 0 & \lambda  & 0 \\ 0 & 0 & \frac 1\lambda\end{pmatrix}, \; \R^3\rtimes \begin{pmatrix} 1 &\star & \star \\ 0 & \cos\theta & -\sin \theta \\ 0 & \sin\theta & \cos\theta\end{pmatrix}.
\end{equation}
Hence, for each family, the translation along the $x$-axis is an automorphism and, of course, the holonomy lies in the highlighted group $\R^3\rtimes(\R^2\rtimes\SL(2,\R))$.

\paragraph{Note}
The existence a non discrete automorphism group may be compared with Ghys' approach in~\cite{Ghys}.
In~\cite{Ghys}, Ghys proved completeness of closed manifolds when an Anosov flow has smooth contracting and expanding line distributions.  He constructs a geometric structure where the Anosov flow becomes a one-parameter family of automorphisms. 

%Recently, Mion-Mouton~\cite{MM} studied a generalization of Ghys theorem when one  relaxes the hypothesis of an Anosov flow to a discrete partially hyperbolic diffeomorphism with smooth contracting and expanding line distributions, a contact structure and a smooth transverse neutral line distribution. In this sOne obtains a  Cartan connection in the geometry of:
%\begin{equation}
%\begin{pmatrix}z\\ y \\ x \end{pmatrix} \rtimes \begin{pmatrix} 1 & x & -y \\ 0 & \lambda & 0 \\ 0 &0& \frac 1\lambda \end{pmatrix} \subset \R^3\rtimes(\R^2\rtimes \SL(2,\R))\subset \R^3\rtimes \SL(3,\R).
%\end{equation}
%(Note that this geometry belongs to one of the families described by Fried and Goldman.)
%The case of Ghys~\cite{Ghys} corresponds to a constant curvature manifold. The flat case is not trivial (there are flat nil-affine manifolds with non trivial partially hyperbolic diffeomorphisms) and the classification was studied by Mion-Mouton~\cite{MM}.

\paragraph{Acknowledgment}The author is grateful to Elisha Falbel and Charles Frances for continuous discussions.

\tableofcontents

\section{Closed flat affine manifolds}

In this short section, we give elementary tools for the treatment of closed flat affine manifolds. 
Notably what one can deduce from the hypothesis of the holonomy centralizing a transformation when the manifold is closed.

\subsection{Holonomy centralizing a transformation}

An important general result about flat affine manifolds with parallel volume is the following.
\begin{theorem}[Goldman-Hirsch~\cite{GH}]
Let $M$ be a connected closed flat affine manifold with parallel volume. Then the holonomy of $M$ is \emph{irreducible}: it cannot preserve any affine subspace.
\end{theorem}

From this, we can deduce what kind of affine transformations can be centralized by the holonomy. For instance, it is clear that this implies that the centralized transformation cannot have fixed points.

\begin{lemma}
Two affine transformations $c+f(x)$ and $b+g(x)$ commute if, and only if, $[f,g]=\id$ and
\begin{equation}
c-g(c) + f(b)-b = 0.
\end{equation}
\end{lemma}

\begin{proof}
It is deduced from:
\begin{align}
(c+f(x))\circ (b+g(x)) &= c + f(b) + fg(x)\\
(b+g(x))\circ(c+f(x)) &= b + g(c) + gf(x).\qedhere
\end{align}
\end{proof}

\begin{lemma}
For any affine transformation $\phi = b + g(x)$, one can assume up to a conjugation by a translation, that $b$ belongs to $\cE_1(g)$, the eigenspace of $g$ associated to the eigenvalue $1$.
\end{lemma}
\begin{proof}
We look for a fixed point (which should give $b=0$ when $\cE_1(g)=\{0\}$).
\begin{align}
b + g(x) &= x\\
\iff (g-\id)(x) &= -b
\end{align}
First assume that $\cE_1(g)=\{0\}$. Then let $b'=(g-\id)^{-1}(b)$. It verifies $(g-\id)(-b') = g(-b') +b' = -b$. Hence, conjugating $\phi$ by the translation by $b'$ gives indeed:
\begin{equation}
b + g(x-b') + b' =  g(x).
\end{equation}
In general, by choosing a supplement $\cE_1(g)\oplus F = \R^n$, one can consider $\phi$ acting on the quotient space $\R^n/\cE_1(g) \simeq F$. It acts as an affine translation since $g$ acts as the identity on $\cE_1(g)$:
\begin{align}
\forall u\in \cE_1(g), \; \phi(x+u) &= b + g(x)+g(u) \\
&= b + g(x) + u = \phi(x) + u.
\end{align}
Therefore up to a translation in $F$, $b$ can be let in $\cE_1(g)$.
\end{proof}

\begin{proposition}\label{prop-transl}
Let $M$ be a connected closed affine flat manifold with parallel volume.
Let $\phi(x) = b + g(x)$ be a non-trivial transformation commuting with the holonomy of $M$ with $b\in \cE_1(g)$. Then $b\neq 0$.
\end{proposition}
\begin{proof}
 Otherwise, $\phi(x) = g(x)$ and for any holonomy transformation $\gamma\in \Gamma$, we would have 
 \begin{equation}
 \forall u\in \cE_1(g), \; \phi\gamma(u)=\gamma \phi(u) = \gamma(u).
 \end{equation}
 Therefore $\gamma$ would preserve $\cE_1(g)$ which is a strict subspace of $\R^n$ since $\phi$ is not trivial. It is in contradiction with the irreducibility of the holonomy by Goldman-Hirsch~\cite{GH}.
\end{proof}

\subsection{Visibility and convexity}

We will assume that $M$ is a flat affine manifold. We denote by $D$ its developing map (which is determined up to an affine conjugation of $\R^3$).

In flat affine geometry, a powerful notion is the convexity of sets. Because affine segments are determined by their endpoints, convexity allows global arguments. 
Classic references are  for instance Benzécri~\cite{Benzecri} and Koszul~\cite{Koszul}. Two references closer to our tools and  our point of view are  Fried~\cite{Fried} and Carrière~\cite{Carriere}.
%citations
%dev

\begin{definition}
Let $M$ be a flat affine manifold. 
\begin{itemize}
\item A \emph{geodesic} $\gamma\colon \iI \to \widetilde M$ is a curve such that $D(\gamma)$ is an affine segment of the form $D(\gamma)(t) = x + tv$.
\item A \emph{convex set} $C\subset \widetilde M$ is a set where $D|_C$ is injective and $D(C)$ is a convex of $\R^n$. (Note that any couple $(x,y)\in C^2$ is related by a geodesic when $C$ is convex.)
\item A \emph{convex set} $C\subset M$ is a set $C\subset U$ with $U$ trivialisable, such that  any lift $C_i\subset \widetilde M$ of $C$ is convex. (Note that, this is true simultaneously for every lift since the affine group preserves the convex sets of $\R^n$.)
\item Let $x\in\widetilde M$. A point $z\in \R^n$ is \emph{visible} from $x$ if  there exists a geodesic curve $\gamma\colon\iI \to \widetilde M$ such that $\gamma(0)=x$ and $D(\gamma(1)) = z$.
\item Let $x\in \widetilde M$. A set $S\subset \R^n$ is \emph{visible} from $x$ if every point of $S$ is visible from $x$.
\item Let $S\subset \widetilde M$ be a set. A set $C\subset \R^n$ is \emph{completely visible} from $S$ if $C$ is visible from any point of $S$.
\end{itemize}
\end{definition}

\begin{lemma}
Let $C\subset \R^n$ be a set visible from a point $x\in \widetilde M$. Consider the set $B\subset \widetilde M$ given by the endpoints of  all the geodesics started at $x$ and with endpoints developed in $C$. The developing map is injective on $B$.
\end{lemma}
\begin{proof}
Two points $a,b\in B$ correspond to two geodesics $\gamma_1,\gamma_2$ started at $x$ and ending at $a$ and $b$ respectively. Both $D(\gamma_1)$ and $D(\gamma_2)$ start at $D(x)$, hence if they have same endpoints, namely $D(a)$ and $D(b)$, then the two affine segments are equal, hence $\gamma_1=\gamma_2$ and $a=b$.
\end{proof}

Thus, when $C$ is visible, it suffices for it to be convex in $\R^n$ in order for $B$ to be convex in $\widetilde M$.

\begin{lemma}
If $C_1,C_2\subset\widetilde M$ are two convex sets with $C_1\cap C_2\neq\emptyset$, then the developing map is injective on $C_1\cup C_2$.
\end{lemma}
\begin{proof}
Let $c_1\in C_1$ and $c_2\in C_2$ with $D(c_1)=D(c_2)$. Let $x\in C_1\cap C_2$. Then the geodesic from $x$ to $c_1$ and the geodesic from $x$ to $c_2$ are sent to the same geodesic in $\R^n$. They must be equal in $\widetilde M$ since both are in $C_1$ since their images are  both in $D(C_1)$.
\end{proof}

Note that when  $A,B\subset \R^n$ are two convexes, then 
\begin{equation}
A+B = \left\{a+b\,\lvert \forall a\in A, \forall b\in B\right\}
\end{equation}
is a convex set of $\R^n$.

\begin{lemma}\label{lem-sumconv}
Let $C\subset \widetilde M$ be convex and $F\subset \R^n$  a linear subspace. If $D(x)+F$ is visible from any $x\in C$, then there exists a convex $V$ such that $C\subset V$ and $D(V) = D(C) + F$.
\end{lemma}
\begin{proof}
For every $x\in C$, $D(x)+F$ is visible and convex, hence we denote $V_x\ni x$ the convex antecedent in $\widetilde M$. Consider the union  $V = \bigcup V_x$. If the developing map is injective on $V$ then it is the desired convex since the image is $D(C)+F$ by construction. 

Let $v_x\in V_x$ and $v_y\in V_y$ such that $D(v_x)=D(v_y)$. Then $D(v_x) = D(x) +u_1$ and $D(v_y) = D(y)+u_2$ for $u_1,u_2\in F$. They verify by hypothesis 
\begin{equation}
D(x)-D(y) = u_2 - u_1.
\end{equation}

We show that $v_y\in V_x$. By convexity of $V_x$ it will conclude. The previous equation implies that for any $t\in \iI$:
\begin{align}
D(x) - t (D(x)-D(y)) & = (1-t)D(x) + tD(y)\\
& = D(x)  - t(u_2-u_1)
\end{align}
Hence the geodesic $\gamma$ from $x$ to $y$ is developed to $D(\gamma)(t) = D(x) - t(u_2-u_1)$. But $-t(u_2-u_1)\in F$ hence $\gamma$ is completely included in $V_x$. Thus $y\in v_x$, which concludes.
\end{proof}

This lemma motivates the following definition.

\begin{definition}
Let $S\subset \widetilde M$ be a set and $F\subset \R^n$ be a linear subspace. Then $F$ is  \emph{(completely) integrable} from $S$ if $D(x) + F$ is visible for any $x\in S$.
\end{definition}

\section{Completeness of some fibered affine structures}

In this section, we show what will be the final step of theorem~\ref{thm-3fold}. We will have to show completeness of manifolds with particular holonomies.
We show this completeness in a greater generality.

\begin{definition}
Let $H\subset \GL(n,\R)$ be a subgroup. Its \emph{discompacity} is defined as follows. For any sequence $h_n\in H$, the singular values of $h_n$, that is to say the eigenvalues of $h_n^Th_n$, might tend to zero. The discompacity of $H$ is the maximal number of singular values of $h_n$ that can tend to zero for the arbitrary choice of a sequence $h_n\in H$.
\end{definition}

\begin{definition}
Let $G\subset \SL(n,\R)$ be a subgroup with matrices of the form:
\begin{equation}
\begin{pmatrix}
U & L \\
0 & D
\end{pmatrix}
=
\begin{pmatrix}
\begin{matrix}
1 & \star & \star \\
 0& \ddots & \star\\
0 &0& 1
\end{matrix}
&
\begin{matrix}
\star & \dots & \star \\
\vdots && \vdots \\
\star & \dots & \star
\end{matrix}\\
\begin{matrix}
0&\dots & 0 \\
\vdots &\ddots &\vdots \\
0 & \dots & 0
\end{matrix}
&D
\end{pmatrix}
\end{equation}
where $U$ is unipotent upper-triangular in $\SL(n-p,\R)$ for a fixed $p< n$, $L$ is any matrix in $\GL((n-p)\times p,\R)$ and $D\in H$ is a matrix belonging to a subgroup $H\subset \SL(p,\R)$ with  $0$ or $1$ discompacity. 
We say that $\R^n\rtimes G$ is \emph{unipotent co-discompacity one fibered}.
\end{definition}

\paragraph{Note}
Affine structures with such holonomies carry a natural affine fiber bundle. Indeed, $\R^{n-p}\times \{0\}_{\R^p}$ is stabilized by $G$.

\paragraph{Example}
Our favorite example is $G=\R^2\rtimes\SL(2,\R)$:
\begin{align}
\begin{pmatrix}
1 & \begin{matrix}\R & \R\end{matrix} \\ 
\begin{matrix}0\\0\end{matrix} & \SL(2,\R)
\end{pmatrix}
\end{align}
The geometry will correspond to a circle fiber bundle over an affine surface.

\paragraph{Total discompacity}
When $D$ belongs to a subgroup with discompacity one   and the upper-part $L$ is non trivial, $G$ might have its discompacity larger than $1$. Hence it is not covered by Carrière's result~\cite{Carriere}.
For instance, in $G = \R^2\rtimes \SL(2,\R)$, 
\begin{equation}
M(\lambda) = \begin{pmatrix}1 & 0& \frac 1{\lambda^2}\\ 0 &\lambda &0\\0&0&\frac 1\lambda\end{pmatrix} 
\end{equation}
has discompacity two when $\lambda\to 0$.

\begin{theorem}\label{thm-unip-disc1}
Let $M$ be a connected closed manifold with a $(\R^n\rtimes G,\R^n)$-structure with $G$ unipotent co-discompacity one fibered, then $M$ is complete.
\end{theorem}

This theorem furnishes new cases to the Markus conjecture since $G$ is not covered by Carrière's nor by Fried-Goldman-Hirsch's results.

The proof is as follows. We first show completeness along $\R^{n-p}\times\{0\}$, the fiber part. Then we show that in $\R^n$, the discompacity one hypothesis allows us to show that the developing map is either a diffeomorphism or a cover onto the complement of a hyperplane. By irreducibility following Goldman-Hirsch, we will deduce  the completeness of the manifold.

\begin{lemma}
The linear subspace $\R^{n-p}\times\{0\}$ is completely integrable from $\widetilde M$.
\end{lemma}

This lemma is about the completeness of unipotent closed manifolds.  Fried-Goldman-Hirsch~\cite[theorem 6.8]{FGH} propose a proof when $p=0$. We state our proof differently in order to obtain the case $p>0$.

\begin{proof}
Start with the direction $\R e_1\subset \R^{n-p}$. Since $e_1$ is fixed by $G$, the vector field in $\R^n$ defined by $V(z) = e_1|_z$ is invariant under $\R^n\rtimes G$. Hence it is covered in $\widetilde M$ by a $\pi_1(M)$-invariant vector field $W$ that projects into $M$. As $M$ is closed, it implies that $W$ is complete and therefore $D(x) + \R e_1$ is visible from $x$ for any $x\in \widetilde M$.

Now assume that $\R^k=\langle e_1,\dots,e_k\rangle \subset \R^{n-p}$ is integrable from any point of $\widetilde M$. By sum of convex sets (see lemma~\ref{lem-sumconv}), to show that $\R^{k+1}$ is integrable, it suffices to consider the direction $\R e_{k+1}$ and show it is integrable. 
Assume that $\gamma$ is geodesic and incomplete at $t=T$ with $\dd D_x\gamma'(0) = e_{k+1}$. That is to say,
\begin{equation}
D(\gamma)(t) = D(x) + t e_{k+1}
\end{equation}
and $\gamma$ has no prolongation by continuity at $t=T$. Then $\gamma$ projects in $M$ to an infinite curve with no continuation at $t=T$.  For simplicity, say $T=1$.

We consider  the associated Fried dynamics~\cite{Fried}: $\pi(\gamma)$
 has an accumulation point by compacity of $M$, say $y\in M$. Let $U$ be a small neighborhood of $y$. Let $t_k \to 1$ such that $\pi(\gamma)(t_k)\in U$,  $\pi(\gamma)(t_k) \to y$ and $\pi(\gamma)$ exits $U$ between the times $t_k$ and $t_{k+1}$.
 Lift $(U,y)$ to $(U_m,y_m)$ such that $\gamma(t_m)\in U_m$.
Consider the transformations $g_{ji}\in\pi_1(M,y)$ such that $g_{ji}$ sends $(U_i,y_i)$ to $(U_j,y_j)$. Since $\pi(\gamma)(t_k) \to y$, we have
\begin{equation}
g_{ji}^{-1}(\gamma(t_j))\to y_i
\end{equation}
and since $\pi(\gamma)$ exits $U$ between two times,
\begin{equation}
g_{ji}^{-1}(\gamma(t_{j+n}))\not\in U_i,
\end{equation}
for any $n>0$.

Denote by $\rho\colon\pi_1(M,y)\to \R^n\rtimes G$ the holonomy morphism.
In the developing map, we can assume, up to schrink $U$, that each $D(U_m)$ is a convex set of $\R^n$.  By convexity and since $D(\gamma)$ is a segment, it implies that
\begin{align}
\rho(g_{ji}^{-1})(D(x) + t_j e_{k+1}) &\to D(y_i)\\
\rho(g_{ji}^{-1})(D(x) + t_{j+n} e_{k+1}) &\not\in D(U_i). 
\end{align}
But, when $n\to \infty$, $D(\gamma)(t_{j+n})$ does converge in $\R^n$ to $D(x)+e_{k+1}$. Hence:
\begin{equation}
\rho(g_{ji}^{-1})(D(x) + e_{k+1}) \not\in D(U_i).
\end{equation}
Indeed, otherwise, the segment $D(\gamma)(t)$ from $t=t_i$ to $t=1$ is completely included in $D(U_i)$ for a large $j>i$ and hence is complete at $t=1$, which is impossible by assumption. 

Now, since $\R^{k}\subset \R^{n-p}$ is assumed to be always integrable, $D(y_i)+ \R^k$ is visible from $y_i$ and more generally for any convex  $C\subset \widetilde M$, $D(C)+\R^k$ is completely visible from $C$.

The sequence $\rho(g_{ji}^{-1})(D(x) + e_{k+1})$ cannot intersect $\rho(g_{ji}^{-1})(D(U_j)) = D(U_i)$ for any $j>i$ large enough. By considering the convex  sets $D(U_j)+\R^k$ visible from $\gamma(t_j)\in U_j$, the point $\rho(g_{ji}^{-1})(D(x) + e_{k+1})$ cannot belong to $\rho(g_{ji})^{-1}(D(U_j)+\R^k)=D(U_i)+\R^k$ either when $j>i$ is large enough.

By writing $\rho(g_{ji})(x) = c_{ji}+f_{ji}(x)$ with $c_{ji}\in \R^n$ and $f_{ji}\in G$ we have:
\begin{align}
\rho(g_{ji}^{-1})(x) &= -f_{ji}^{-1}(c_{ji}) + f_{ji}^{-1}(x) \\
\rho(g_{ji}^{-1}(D(x) + t_j e_{k+1}) &= 
-f_{ji}^{-1}(c_{ji}) + f_{ji}^{-1}(D(x)) + t_j f_{ji}^{-1}(e_{k+1}) \to D(y_i)\\
\rho(g_{ji}^{-1})(D(x) +  e_{k+1}) &= 
-f_{ji}^{-1}(c_{ji}) + f_{ji}^{-1}(D(x)) +  f_{ji}^{-1}(e_{k+1}) \not\in D(U_i)+\R^k.
\end{align}
Therefore, in order for the last sequence to not belong to $D(U_i) + \R^k$, we must have $f_{ji}^{-1}(e_{k+1})$ diverging along a direction  outside $\R^k\subset \R^{n-p}$. But $f_{ji}$ is unipotent, so $f_{ji}^{-1}(e_{k+1}) \in e_{k+1}+\R^k$, which gives a contradiction. Hence $\R e_{k+1}$ is always integrable and we conclude the proof.
\end{proof}

\begin{proof}[Proof of the theorem]
To finish the proof, we follow Carrière's argument~\cite{Carriere}. 

 Let $\Delta$ be a triangle incomplete at his boundary: for $x\in\widetilde M$ we consider $\Delta\subset \R^n$ a maximal triangle where $D(x)$ is a vertex, the  other vertices being $D(y)$ and $D(z)$ that are visible from $x$ but the edge between those last two vertices is not visible from $x$. We can assume that the interior of $\Delta$ is fully visible from $x$.

Parametrize the edge between $D(y)$ and $D(z)$ by $w(t)\in [D(y),D(z)]$ with $t\in [0,1]$. By openness of the visibility we can assume that $w(T)\in [D(y),D(z)]$ is the first invisible point from $x$ in the sense that for $t<T$, $w(t)$ is visible.

The geodesic between $D(x)$ and $w(T)$ is incomplete, so the corresponding geodesic in $\widetilde M$ explodes in finite time. For $v$ such that $D(x)+ t\dd D_x v = w(t)$, we have that $\gamma(t)=x+tv$ is incomplete at $t=1$.
Recall that $\R^{n-p}$ is completely integrable on $\widetilde M$.
Up to consider the antecedent to $\Delta + \R^{n-p}$ and vary $x$, we can assume that $\Delta$ lies in the subspace $\{0\}\times \R^p\subset \R^n$, defined by $e_1=\dots=e_{n-p}=0$, based at $D(x)$. Hence $\dd D_x v$ has vanishing coordinates in $\R^{n-p}\times\{0\}$.

We consider again the associated Fried dynamics $(\{g_{ji}\},\{U_i\})$. Each $U_i$ can again be assumed to be convex. We can prolongate each $U_i$ to the convex developed into $D(U_i)+\R^{n-p}$ by integrability of $\R^{n-p}$. (See lemma~\ref{lem-sumconv}.)

In $\Delta+\R^{n-p}$, the convex sets $D(U_i)+\R^{n-p}$ intersect $D(\gamma)$ at the times $t_j\to 1$ and $D(\gamma)$ exits $D(U_i)+\R^{n-p}$ before entering $D(U_{i+1})+\R^{n-p}$. Note that the holonomy transformations $\rho(g_{ji})$ transforming $D(U_i)$ into $D(U_j)$ transform also $D(U_i) + \R^{n-p}$ to $D(U_j)+\R^{n-p}$ since the linear group $G$ preserves $\R^{n-p}$.

Therefore $D(U_j) + \R^{n-p} = \rho(g_{ji})(D(U_i)+ \R^{n-p})$ degenerate to a hyperplane $H$ of $\R^n$. Indeed, $G$ has discompacity one in the quotient by $\R^{n-p}$ by hypothesis. Therefore the limit $H$ must be tangent or transverse to $\Delta+\R^{n-p}$.
Since $G$ preserves $\R^{n-p}$, $H$ is tangent or transverse to $\Delta\subset \R^p$.
 Both case are impossible by visibility of the interior of $\Delta$ and visibility of $w(s)$ for $s<T$. 
Indeed, otherwise, by taking a limit point in $\overline \Delta$, there would be a complete geodesic developed into $\Delta$ and intersecting infinitely many $U_i$, impossible for a compact curve.

It follows that there are no triangle with incomplete boundary in $\widetilde M$. Hence $\widetilde M$ is convex and is developed in $\R^n$ inside a convex set. By the same construction, this convex $D(\widetilde M)$ is either a half-space bounded by a limit hyperplane $H$ or is a band between two of them which must be parallel. The second case cannot happen since otherwise the distance function to the boundary would give a function on $M$ that has no minimal value, contradicting its compactness. Therefore $D(\widetilde M)$ is a half-space and $H$ must be preserved by the holonomy. It is a contradiction by the fact that the holonomy cannot preserve any proper affine subspace by Goldman-Hirsch~\cite{GH}. 
\end{proof}

\section{Three-manifolds with parallel volume and completeness}

In this section we show:

\begin{theorem}\label{thm-3fold}
Let $M$ be a connected closed $(\R^3,\R^3\rtimes \SL(3,\R))$-manifold. If the holonomy of $M$ centralizes a non trivial transformation $\phi\in \R^3\rtimes \SL(3,\R)$, then $M$ is complete.
\end{theorem}

As stated in introduction, we deduce a corollary from interpreting the hypothesis in two ways.

\begin{corollary} \label{cor-1}
Let $M$  be a  connected closed $(\R^3,\R^3\rtimes \SL(3,\R))$-manifold. 
\begin{itemize}
\item If the holonomy of $M$ has non trivial center,   then $M$ is complete.
\item If the automorphism group of $M$ is non discrete, then $M$ is complete.
\end{itemize}
\end{corollary}

\paragraph{Note}An automorphism of $M$ is a diffeomorphism $f$ of $M$ that can be lifted to a diffeomorphism $\widetilde f$ of $\widetilde M$ such that $D(\widetilde f(x)) = \chi(\widetilde f)D(x)$ for a fixed $\chi(\widetilde f)\in \R^3\rtimes \SL(3,\R)$. Let $\widetilde f_n\to\id$  be a sequence of automorphisms. Because $\widetilde f_n$ lifts diffeomorphisms of $M$, they normalize the fundamental group $\pi_1(M)$. But since $\pi_1(M)$ acts on a discrete fiber, the transformations $\widetilde f_n$ must centralize $\pi_1(M)$ once they are small enough. For instance, let $U\subset\widetilde M$ be contractible and $x\in U$. Then, for any $g\in\pi_1(M,x)$, $\widetilde f_n g(x)$ tends indeed to a point in $gU$, so $\widetilde f_n g = g' \widetilde f_n$ implies $g'=g$.

\paragraph{Seifert bundles}
The following corollary was shown for closed Seifert bundles over hyperbolic surfaces (with different methods) by Carrière, Dal'bo and Meignez~\cite{CDM}.  When it is put together with completeness of solvable manifolds by Goldman-Hirsch~\cite{GH2} one can obtain  the corollary.   We give a different proof that does not rely on this distinction.

\begin{corollary}\label{cor-2}
Let $M$ be a connected closed  Seifert  bundle over an orientable surface. Then $M$ verifies the Markus conjecture: if $M$ has a flat affine structure with parallel volume then $M$ is complete.
\end{corollary}

\begin{proof}[Proof of the corollary]
Following Goldman-Hirsch~\cite{GH} and Smillie~\cite{Smillie}, a closed affine manifold with parallel volume    of dimension $n$ has a holonomy group that has cohomological dimension $\geq n$. Hence if $M$ is a Seifert bundle, the holonomy of the fiber cannot vanish since otherwise the holonomy would have cohomological dimension $2<3$. But the fiber corresponds to a central subgroup, hence we can apply the theorem.
\end{proof}

With theorem~\ref{thm-unip-disc1}, we only need to show the following proposition to prove theorem~\ref{thm-3fold}.

\begin{proposition}
Let $M$ be a connected closed  $(\R^3,\R^3\rtimes \SL(3,\R))$-manifold. Assume that the holonomy of $M$ centralizes a non trivial transformation $\phi\in \R^3\rtimes \SL(3,\R)$.
Up to a global affine conjugation, the linear holonomy of $M$ takes its values in $\R^2\rtimes \SL(2,\R)$:
\begin{equation}
\R^2\rtimes\SL(2,\R) = \begin{pmatrix}
1 & \begin{matrix}\R & \R\end{matrix} \\ 
\begin{matrix}0\\0\end{matrix} & \SL(2,\R)
\end{pmatrix}
\end{equation}
\end{proposition}

\begin{proof}
Up to a global affine conjugation, by proposition~\ref{prop-transl}, one can assume that $\phi(x) = e_1 + g(x)$ with $e_1 \in \cE_1(g)\neq \{0\}$. 

Let $\gamma(x)=c+f(x)$ be a holonomy transformation. By commutativity, $f(\cE_1(g)) = \cE_1(g)$ and recall we have the equations $[f,g]=e$ and $c-g(c)+f(e_1)-e_1=0$. 

Note that in dimension $3$ with determinant $1$, either $1$ is  a simple eigenvalue or  it is the only (complex) eigenvalue of $g$.
\begin{enumerate}
\item If $g = \id$ then the  equation $c-g(c)+f(e_1)-e_1=0$ reduces to $f(e_1)=e_1$, hence $f\in \R^2\rtimes \SL(2,\R)$.
\item If $1$ is a simple eigenvalue of $g$, then we can complete $e_1$ by a basis $(e_1,e_2,e_3)$ such that
\begin{equation}
g=\begin{pmatrix}
1 & \begin{matrix}0 & 0\end{matrix} \\
\begin{matrix}0\\0\end{matrix} & A\in \SL(2,\R)
\end{pmatrix}
\end{equation}
with $A$ that does not have $1$ as eigenvalue. By writing
\begin{equation}
f=\begin{pmatrix}
\alpha & {}^tv \\
w & B\in \GL(2,\R)
\end{pmatrix},
\end{equation}
the equation $[f,g]=e$ gives ${}^t vA - {}^t v = 0$ and $w- Aw=0$. Since $1$ is not an eigenvalue of $A$, we have $v=w=0$. Hence we have $f(e_1)=\alpha e_1$. But $g(c)-c$ has $0$ for first coordinate, hence $f(e_1)-e_1=(\alpha-1)e_1=0$, hence $\alpha=1$. Therefore $f\in \R^2\rtimes \SL(2,\R)$.

\item Now we assume that $g\neq \id$ but is unipotent. There are three cases to consider by completing $e_1$ into a basis $(e_1,e_2,e_3)$:
\begin{align}
g_1 &=\begin{pmatrix}
 1 & 0 & 0 \\
 0 & 1 & 1 \\
 0 & 0 & 1
\end{pmatrix},\\
g_2 &=
\begin{pmatrix}
 1 & 1 & 0 \\
 0 & 1 & 1 \\
 0 & 0 & 1
\end{pmatrix},\\
g_3 &=
\begin{pmatrix}
 1 & 0& 1 \\
 0 & 1 & 0 \\
 0 & 0 & 1
\end{pmatrix}.
\end{align}
We write:
\begin{equation}
f  = \begin{pmatrix}
\alpha & b & w\\
x& \beta & v \\
y & z & \gamma
\end{pmatrix}.
\end{equation}
By a direct computation:
\begin{align}
[f,g_1] = e&\implies f =
\begin{pmatrix}
\frac 1{\beta^2} & 0 & w\\
x& \beta & v \\
0 & 0 & \beta
\end{pmatrix},\\
[f,g_2] = e &\implies f=
\begin{pmatrix}
1 & b & w\\
0& 1 & b \\
0& 0 &1
\end{pmatrix}\in\R^2\rtimes \SL(2,\R),\\
[f,g_3] = e &\implies f=
\begin{pmatrix}
\alpha & b & w\\
0& \frac 1{\alpha^2}& v \\
0 & 0 & \alpha
\end{pmatrix}.
\end{align}

Thus the case $g=g_2$ is clear.

When $g=g_1$ then $f(e_1)-e_1$ has for first coordinate $\frac 1{\beta^2}-1$ but $g(c)-c$ has for first coordinate $0$. Hence $\beta=1$. By changing the basis from $(e_1,e_2,e_3)$ to $(e_2,e_1,e_3)$, which does not depends on $f$, we get $f\in \R^2\rtimes \SL(2,\R)$.

When $g=g_3$, denote $c=(c_1,c_2,c_3)$, then $g(c)-c=c_3e_1$ and $f(e_1)-e_1 = (\alpha-1)e_1$. Hence $c_3=\alpha-1$. Therefore, the holonomy $\Gamma$ preserves the hyperplane $z+1=0$. But by irreducibility of the holonomy by Goldman-Hirsch~\cite{GH}, this cannot happen so $g$ can never be equal to $g_3$.\qedhere
\end{enumerate}
\end{proof}

\begin{proof}[Proof of the theorem~\ref{thm-3fold}]
Since $\R^2\rtimes \SL(2,\R)$ is unipotent co-discompacity one fibered, we apply theorem~\ref{thm-unip-disc1}.
\end{proof}

\printbibliography

\end{document}